\documentclass{article}

\usepackage{amsmath}
\usepackage{amsthm}
\usepackage{alltt}
\usepackage[pdftex]{hyperref}

\newtheorem{theorem}{Theorem}[section]

\numberwithin{equation}{section}

\newcommand{\res}[1]
{\textrm{\rm Res}_{#1}}
\newcommand{\cycle}[2]
{\genfrac{[}{]}{0pt}{}{#1}{#2}}

\title{Some Weighted Generalized Fibonacci Number Summation Identities, Part 1}

\author{M.J. Kronenburg}
\date{}

\begin{document}

\maketitle

\begin{abstract}
The Fibonacci number is the residue of a rational function,
from which follows that Fibonacci number summation identities can be derived with the
integral representation method, a method also used to derive combinatorial identities.
A number of weighted generalized Fibonacci number summation identities
are derived this way.
From these identities some infinite series, generating functions
and convolution identities are obtained.
In addition, some weighted generalized Fibonacci number summation identities
with binomial coefficients are derived.
Many examples of both types of summation identities are provided.
\end{abstract}


\section{Weighted generalized Fibonacci number\\ summation identities}

Let $F_n$ be the Fibonacci number, $L_n$ be the Lucas number, and $G_n$ be the generalized Fibonacci number
for which $G_{n+2}=G_{n+1}+G_n$ with any seed $G_0$ and $G_1$ \cite{BQ03,CW00,V89}.
The following weighted generalized Fibonacci number summation identities are derived.
\begin{equation}\label{res1}
 \Delta = 1-L_b x+(-1)^b x^2
\end{equation}
\begin{equation}\label{pnmres}
 P(v,w) = \sum_{k=0}^v \binom{v}{k} (-1)^{(b+1)k} x^k G_{a+bw-bk}
\end{equation}
\begin{equation}\label{res2}
 \sum_{k=0}^n x^k G_{a+bk} = -\frac{1}{\Delta}[ x^{n+1} P(1,n+1) - P(1,0) ]
\end{equation}
\begin{equation}\label{res3}
 \sum_{k=0}^n x^k k G_{a+bk} = -\frac{1}{\Delta}(n+1)x^{n+1} P(1,n+1)
  - \frac{1}{\Delta^2} [ x^{n+2} P(2,n+2) - x P(2,1) ]
\end{equation}
\begin{equation}\label{res4}
\begin{split}
  \sum_{k=0}^n x^k k^2 G_{a+bk} = & -\frac{1}{\Delta}(n+1)^2 x^{n+1} P(1,n+1) \\
  & - \frac{1}{\Delta^2} [ (2n+3) x^{n+2} P(2,n+2 ) - x P(2,1) ] \\
  & - \frac{2}{\Delta^3} [ x^{n+3} P(3,n+3) - x^2 P(3,2) ]  \\
\end{split}
\end{equation}
\begin{equation}\label{res5}
\begin{split}
  \sum_{k=0}^n x^k k^3 G_{a+bk} = & -\frac{1}{\Delta}(n+1)^3 x^{n+1} P(1,n+1) \\
  & - \frac{1}{\Delta^2} [ (3n^2+9n+7) x^{n+2} P(2,n+2 ) - x P(2,1) ] \\
  & - \frac{6}{\Delta^3} [ (n+2) x^{n+3} P(3,n+3) - x^2 P(3,2) ]  \\
  & - \frac{6}{\Delta^4} [  x^{n+4} P(4,n+4) - x^3 P(4,3) ]  \\
\end{split}
\end{equation}
\begin{equation}\label{res6}
\begin{split}
  \sum_{k=0}^n x^k k^4 G_{a+bk} = & -\frac{1}{\Delta}(n+1)^4 x^{n+1} P(1,n+1) \\
  & - \frac{1}{\Delta^2} [ (4n^3+18n^2+28n+15) x^{n+2} P(2,n+2 ) - x P(2,1) ] \\
  & - \frac{2}{\Delta^3} [ (6n^2+24n+25) x^{n+3} P(3,n+3) - 7 x^2 P(3,2) ]  \\
  & - \frac{12}{\Delta^4} [ (2n+5) x^{n+4} P(4,n+4) - 3 x^3 P(4,3) ]  \\
  & - \frac{24}{\Delta^5} [ x^{n+5} P(5,n+5) - x^4 P(5,4) ] \\
\end{split}
\end{equation}
\begin{equation}\label{res7}
\begin{split}
  \sum_{k=0}^n & x^k k^5 G_{a+bk} = -\frac{1}{\Delta}(n+1)^5 x^{n+1} P(1,n+1) \\
  & - \frac{1}{\Delta^2} [ (5n^4+30n^3+70n^2+75n+31) x^{n+2} P(2,n+2 ) - x P(2,1) ] \\
  & - \frac{10}{\Delta^3} [ (2n^3+12n^2+25n+18) x^{n+3} P(3,n+3) - 3 x^2 P(3,2) ]  \\
  & - \frac{30}{\Delta^4} [ (2n^2+10n+13) x^{n+4} P(4,n+4) - 5 x^3 P(4,3) ]  \\
  & - \frac{120}{\Delta^5} [ (n+3) x^{n+5} P(5,n+5) - 2 x^4 P(5,4) ] \\
  & - \frac{120}{\Delta^6} [ x^{n+6} P(6,n+6) - x^5 P(6,5) ] \\
\end{split}
\end{equation}

\section{The Fibonacci Number is a Residue}

The residue of a function $f(x)$ at a pole $x_p$ of order $k$ is \cite{CB84,K08}:
\begin{equation}
 \res{x=x_p} f(x) = \frac{1}{(k-1)!} D_x^{k-1} [ (x-x_p)^k f(x) ]|_{x=x_p}
\end{equation}
where $D_x^n f(x)|_{x=x_p}$ is the n-th derivative of $f(x)$ at $x=x_p$.
From this follows that when $f(x)$ does not have a pole at $x=x_p$, then:
\begin{equation}\label{respdef}
 \res{x=x_p} \frac{f(x)}{(x-x_p)^k} = \frac{1}{(k-1)!} D_x^{k-1} f(x)|_{x=x_p}
\end{equation}
and this residue is zero when $k\leq 0$.
A definition that is needed here is that the residue $\res{x}$ is the sum of the residues
over all poles:
\begin{equation}
 \res{x}f(x) = \sum_{p=1}^{npoles} \res{x=x_p}f(x)
\end{equation}
This definition is not only needed for Fibonacci summation identities but also
for certain combinatorial identities.
When $f(x)$ is a constant, then residue (\ref{respdef}) is only
nonzero when $k=1$. Then using:
\begin{equation}
 (1+x)^n = \sum_{k=0}^n \binom{n}{k} x^k
\end{equation}
it is clear that:
\begin{equation}\label{binomdef}
 \binom{n}{k} = \res{x} \frac{(1+x)^n}{x^{k+1}}
\end{equation}
From this identity many combinatorial identities can be derived \cite{E84}.\\
The Fibonacci and Lucas numbers can be defined by their Binet formulas,
where $\phi$ is the golden ratio:
\begin{equation}
 \phi=\frac{1}{2}(1+\sqrt{5})
\end{equation}
\begin{equation}\label{fibodef}
 F_n = \frac{1}{\sqrt{5}} ( \phi^n - (1-\phi)^n )
\end{equation}
\begin{equation}\label{lucasdef}
 L_n = \phi^n + (1-\phi)^n
\end{equation}
The Fibonacci number is also the residue of a rational function.
\begin{theorem}
\begin{equation}\label{fibdef}
 F_n = \res{x} \frac{(1+x)^n}{(x+\phi)(x-1/\phi)}
\end{equation}
\end{theorem}
\begin{proof}
Using $1+1/\phi=\phi$ and $\phi+1/\phi=\sqrt{5}$:
\begin{equation}
\begin{split}
  & \res{x} \frac{(1+x)^n}{(x+\phi)(x-1/\phi)} \\
 = & \res{x=1/\phi} \frac{(1+x)^n}{(x+\phi)(x-1/\phi)} + \res{x=-\phi} \frac{(1+x)^n}{(x+\phi)(x-1/\phi)} \\
 = & \frac{(1+1/\phi)^n}{\phi+1/\phi} + \frac{(1-\phi)^n}{-\phi-1/\phi} \\
 = & \frac{1}{\sqrt{5}} ( \phi^n - (1-\phi)^n ) \\
 = & F_n
\end{split}
\end{equation}
\end{proof}
Many Fibonacci summation identities can now be derived
using the geometric series.

\section{The Geometric Series}

The finite geometric series below is easily proved with induction on $n$:
\begin{equation}\label{fingeom0}
 \sum_{k=0}^n x^k = \frac{x^{n+1}-1}{x-1}
\end{equation}
This identity is true for $n=0$, and when it is true for $n$, then the identity for $n+1$ becomes:
\begin{equation}
 \sum_{k=0}^{n+1} x^k = \sum_{k=0}^n x^k + x^{n+1} = \frac{x^{n+1}-1}{x-1} + \frac{x^{n+1}(x-1)}{x-1}
 = \frac{x^{n+2}-1}{x-1}
\end{equation}
which makes it true for $n+1$, so it is true for all $n$.
For evaluating the generalized Fibonacci summation identities,
the following modified finite geometric series is needed:
\begin{equation}
 S(m) = \sum_{k=0}^n k^m x^k
\end{equation}
The following follows from (\ref{fingeom0}):
\begin{equation}
 \sum_{k=0}^n x^k \prod_{j=0}^{m-1} (k-j) = x^m \sum_{k=0}^n x^{k-m} \prod_{j=0}^{m-1} (k-j) = x^m D_x^m \frac{x^{n+1}-1}{x-1}
\end{equation}
Now the following is used:
\begin{equation}
 \prod_{j=0}^{m-1}(k-j) = \sum_{j=0}^{m}(-1)^{m-j}\cycle{m}{j}k^j
\end{equation}
where $\cycle{a}{b}$ is the Stirling number of the first kind \cite{GKP94}.
Then by changing the order of summation it is clear that:
\begin{equation}
 S(m) = x^m D_x^m \frac{x^{n+1}-1}{x-1} - \sum_{k=0}^{m-1} (-1)^{m-k} \cycle{m}{k} S(k)
\end{equation}
The corresponding Mathematica$^{\textregistered}$ \cite{W03} program:
\begin{alltt}
S[0]=(x^(n+1)-1)/(x-1);
S[m_]:=S[m]=Collect[x^m D[(x^(n+1)-1)/(x-1),\{x,m\}]
 -Sum[StirlingS1[m,k]S[k],\{k,0,m-1\}],x-1,Simplify]
\end{alltt}
With this computer program the following modified finite geometric series
can be easily computed:
\begin{equation}\label{fingeom1}
 \sum_{k=0}^n k x^k = \frac{(n+1)x^{n+1}}{x-1} - \frac{x^{n+2}-x}{(x-1)^2}
\end{equation}
\begin{equation}\label{fingeom2}
 \sum_{k=0}^n k^2 x^k = \frac{(n+1)^2 x^{n+1}}{x-1} - \frac{(2n+3)x^{n+2}-x}{(x-1)^2}
  + 2 \frac{x^{n+3}-x^2}{(x-1)^3}
\end{equation}
\begin{equation}\label{fingeom3}
\begin{split}
 \sum_{k=0}^n k^3 x^k = & \frac{(n+1)^3 x^{n+1}}{x-1} - \frac{(3n^2+9n+7)x^{n+2}-x}{(x-1)^2} \\
   & + 6 \frac{(n+2)x^{n+3}-x^2}{(x-1)^3} - 6 \frac{x^{n+4}-x^3}{(x-1)^4} \\
\end{split}
\end{equation}
\begin{equation}\label{fingeom4}
\begin{split}
 \sum_{k=0}^n & k^4 x^k = \frac{(n+1)^4 x^{n+1}}{x-1} - \frac{(4n^3+18n^2+28n+15)x^{n+2}-x}{(x-1)^2} \\
   & + 2 \frac{(6n^2+24n+25)x^{n+3}-7x^2}{(x-1)^3} - 12 \frac{(2n+5)x^{n+4}-3x^3}{(x-1)^4} + 24 \frac{x^{n+5}-x^4}{(x-1)^5} \\
\end{split}
\end{equation}
\begin{equation}\label{fingeom5}
\begin{split}
 & \sum_{k=0}^n k^5 x^k = \frac{(n+1)^5 x^{n+1}}{x-1} - \frac{(5n^4+30n^3+70n^2+75n+31)x^{n+2}-x}{(x-1)^2} \\
   & + 10 \frac{(2n^3+12n^2+25n+18)x^{n+3}-3x^2}{(x-1)^3} - 30 \frac{(2n^2+10n+13)x^{n+4}-5x^3}{(x-1)^4} \\
   &  + 120 \frac{(n+3)x^{n+5}-2x^4}{(x-1)^5} - 120 \frac{x^{n+6}-x^5}{(x-1)^6} \\
\end{split}
\end{equation}
The corresponding (modified) infinite geometric series is obtained by taking $|x|<1$ and letting $n\to\infty$:
\begin{equation}\label{infgeom0}
 \sum_{k=0}^{\infty} x^k = \frac{-1}{x-1}
\end{equation}
\begin{equation}\label{infgeom1}
 \sum_{k=0}^{\infty} k x^k = \frac{x}{(x-1)^2}
\end{equation}
\begin{equation}\label{infgeom2}
 \sum_{k=0}^{\infty} k^2 x^k = \frac{x}{(x-1)^2} - \frac{2x^2}{(x-1)^3}
\end{equation}
\begin{equation}\label{infgeom3}
 \sum_{k=0}^{\infty} k^3 x^k = \frac{x}{(x-1)^2} - \frac{6x^2}{(x-1)^3} + \frac{6x^3}{(x-1)^4}
\end{equation}
\begin{equation}\label{infgeom4}
 \sum_{k=0}^{\infty} k^4 x^k = \frac{x}{(x-1)^2} - \frac{14x^2}{(x-1)^3} + \frac{36x^3}{(x-1)^4} - \frac{24x^4}{(x-1)^5}
\end{equation}
\begin{equation}\label{infgeom5}
 \sum_{k=0}^{\infty} k^5 x^k = \frac{x}{(x-1)^2} - \frac{30x^2}{(x-1)^3} + \frac{150x^3}{(x-1)^4} - \frac{240x^4}{(x-1)^5} + \frac{120x^5}{(x-1)^6}
\end{equation}

\section{A Few Simple Examples}

A few simple examples using the above formulas are given,
demonstrating that the integral representation method can be used
both for combinatorial and for Fibonacci summation identities.
As a simple combinatorial identity the following sum is evaluated:
\begin{equation}
 \sum_{k=0}^n k \binom{n}{k} = \sum_{k=0}^n k \res{x} \frac{(1+x)^n}{x^{k+1}}
\end{equation}
Because the summand is zero when $k>n$ the modified infinite geometric series (\ref{infgeom1}) can be used,
resulting in:
\begin{equation}
 \sum_{k=0}^{\infty}  k \res{x} \frac{(1+x)^n}{x^{k+1}} = \res{x} \frac{(1+x)^n}{x} \sum_{k=0}^{\infty} k (\frac{1}{x})^k
 = \res{x} \frac{(1+x)^n}{(x-1)^2}
\end{equation}
With (\ref{respdef}) the result is:
\begin{equation}
 \sum_{k=0}^n k \binom{n}{k} = n 2^{n-1}
\end{equation}
As a simple Fibonacci summation identity the following sum is evaluated:
\begin{equation}
 \sum_{k=0}^n F_{p+k} = \sum_{k=0}^n \res{x} \frac{(1+x)^{p+k}}{(x+\phi)(x-1/\phi)}
\end{equation}
In this case the summand is not zero for $k>n$, and therefore the finite geometric series (\ref{fingeom0}) must be used:
\begin{equation}
 \res{x} \frac{(1+x)^p}{(x+\phi)(x-1/\phi)} \sum_{k=0}^n (1+x)^k
 = \res{x} \frac{(1+x)^{p+n+1}-(1+x)^p}{x(x+\phi)(x-1/\phi)}
\end{equation}
There are three poles, but the residue for $x=0$ is zero as the numerator for $x=0$ is zero.
Summing the remaining two residues for $x=1/\phi$ and $x=-\phi$, and using $1+1/\phi=\phi$, $\phi+1/\phi=\sqrt{5}$,
and $-1/\phi=1-\phi$, and the definition of the Fibonacci number in (\ref{fibodef}), the result is:
\begin{equation}
 \sum_{k=0}^n F_{p+k} = F_{p+n+2} - F_{p+1}
\end{equation}
As a final simple example, the following identity is computed:
\begin{equation}
 \sum_{k=0}^n \binom{n}{k} F_{p+k} = \sum_{k=0}^n \res{x}\res{y} \frac{(1+x)^n(1+y)^{p+k}}{x^{k+1}(y+\phi)(y-1/\phi)}
\end{equation}
In this case the summand is zero for $k>n$, so the infinite geometric series (\ref{infgeom0}) can be used:
\begin{equation}
 \res{x}\res{y} \frac{(1+x)^n(1+y)^p}{x(y+\phi)(y-1/\phi)} \sum_{k=0}^{\infty} \left(\frac{1+y}{x}\right)^k
\end{equation}
resulting in:
\begin{equation}
 \res{y}\res{x} \frac{(1+x)^n(1+y)^p}{(y+\phi)(y-1/\phi)(x-(1+y))}
 = \res{y} \frac{(2+y)^n(1+y)^p}{(y+\phi)(y-1/\phi)}
\end{equation}
Now using $1+1/\phi=\phi$, $\phi+1/\phi=\sqrt{5}$, $2+1/\phi=\phi^2$ and $2-\phi=(1-\phi)^2$:
\begin{equation}
 \sum_{k=0}^n \binom{n}{k} F_{p+k} = F_{p+2n}
\end{equation}

\section{Derivation of the Summation Identities}

The following sum is to be evaluated:
\begin{equation}
\begin{split}
 \sum_{k=0}^n x^k F_{a+bk} & = \sum_{k=0}^n x^k \res{y} \frac{(1+y)^{a+bk}}{(y+\phi)(y-1/\phi)} \\
 & = \res{y} \frac{(1+y)^a}{(y+\phi)(y-1/\phi)} \sum_{k=0}^n (x(1+y)^b)^k
\end{split}
\end{equation}
which with (\ref{fingeom0}) becomes:
\begin{equation}\label{eval1}
  \res{y} \frac{(1+y)^a((x(1+y)^b)^{n+1}-1)}{(y+\phi)(y-1/\phi)(x(1+y)^b-1)}
\end{equation}
Taking the residues at $y=1/\phi$ and $y=-\phi$, and using $1+1/\phi=\phi$,
the following theorem is needed.
\begin{theorem}
\begin{equation}\label{resdef1}
 \frac{1}{x\phi^b-1} = \frac{x(F_{b+1}-F_b\phi)-1}{1-L_bx+(-1)^bx^2}
\end{equation}
\begin{equation}\label{resdef2}
 \frac{1}{x(1-\phi)^b-1} = \frac{x(F_{b+1}-F_b(1-\phi))-1}{1-L_bx+(-1)^bx^2}
\end{equation} 
\end{theorem}
\begin{proof}
Given $v$ and $w$ the variables $V$ and $W$ are solved
in the following two equations, using $1-\phi=-1/\phi$:
\begin{equation}
 \frac{1}{v+w\phi^n} = V + W\phi
\end{equation}
\begin{equation}
 \frac{1}{v+w(1-\phi)^n} = V + W(1-\phi) = V - \frac{W}{\phi}
\end{equation}
Taking the reciprocal and then subtracting these two equations
and using $\phi-1/\phi=1$:
\begin{equation}
 w ( \phi^n - (1-\phi)^n ) = \frac{1}{V+W\phi} - \frac{1}{V-W/\phi} 
 = \frac{-W\sqrt{5}}{V^2-W^2+VW}
\end{equation}
and adding:
\begin{equation}
 w ( \phi^n + (1-\phi)^n ) + 2v = \frac{1}{V+W\phi} + \frac{1}{V-W/\phi} 
 = \frac{2V+W}{V^2-W^2+VW} 
\end{equation}
Substituting the definition of the Fibonacci and Lucas numbers (\ref{fibodef}) and (\ref{lucasdef}):
\begin{equation}
 F_n = \frac{-W}{w(V^2-W^2+VW)}
\end{equation}
\begin{equation}
 wL_n + 2v = \frac{2V+W}{V^2-W^2+VW}
\end{equation}
Dividing these two equations and using $F_n+L_n=2F_{n+1}$:
\begin{equation}
 \frac{V}{W} = z = -\frac{w(F_n+L_n)+2v}{2wF_n} = -\frac{F_{n+1}+v/w}{F_n}
\end{equation}
The following is easily checked:
\begin{equation}
 V^2-W^2+VW = W^2 (z^2+z-1)
\end{equation}
and the $V$ and $W$ become:
\begin{equation}
 W = \frac{-1}{w(z^2+z-1)F_n}
\end{equation}
\begin{equation}
 V = z W
\end{equation}
Now $z^2+z-1$ is easily evaluated:
\begin{equation}
\begin{split}
 z^2+z-1 & = \frac{(F_{n+1}+v/w)^2-F_n(F_{n+1}+v/w)-F_n^2}{F_n^2} \\
  & = \frac{F_{n+1}^2-F_nF_{n+1}-F_n^2 + (2F_{n+1}-F_n)v/w + (v/w)^2}{F_n^2}
\end{split}
\end{equation}
The first term in the numerator can be simplified using Cassini's identity:
\begin{equation}
 F_{n+1}^2-F_nF_{n+1}-F_n^2 = F_{n+1}F_{n-1}-F_n^2 = (-1)^n 
\end{equation}
and as before $2F_{n+1}-F_n=L_n$ which results in:
\begin{equation}
 z^2+z-1 = \frac{(v/w)^2+L_nv/w+(-1)^n}{F_n^2}
\end{equation}
and now $V$ and $W$ are solved:
\begin{equation}
 W = \frac{-wF_n}{v^2+L_nvw+(-1)^nw^2}
\end{equation}
\begin{equation}
 V = zW = \frac{wF_{n+1}+v}{v^2+L_nvw+(-1)^nw^2}
\end{equation}
Taking $v=-1$, $w=x$ and $n=b$, the theorem is proved.
\end{proof}
This result can be simplified with the following theorem.
\begin{theorem}
\begin{equation}\label{ressimp1}
  F_{b+1} - F_b \phi = (-1)^b \phi^{-b}
\end{equation}
\begin{equation}\label{ressimp2}
  F_{b+1} - F_b (1-\phi) = (-1)^b (1-\phi)^{-b}
\end{equation}
\end{theorem}
\begin{proof}
Using $2\phi-1=\sqrt{5}$ and $1-\phi=-1/\phi$:
\begin{equation}
\begin{split}
 & F_{b+1} - F_b \phi \\
 = & \frac{1}{\sqrt{5}} [ \phi^{b+1} - (1-\phi)^{b+1} - ( \phi^b - (1-\phi)^b ) \phi ] \\
 = & \frac{1}{\sqrt{5}} [ \phi (1-\phi)^b - (1-\phi)^{b+1} ] \\
 = & \frac{1}{\sqrt{5}} (1-\phi)^b ( 2\phi - 1 ) \\
 = & (1-\phi)^b \\
 = & (-1)^b \phi^{-b} \\
\end{split}
\end{equation}
The proof of the second identity is similar.
\end{proof}
Now taking the residue for $y=1/\phi$ in (\ref{eval1}) 
and using $1+1/\phi=\phi$ and $\phi+1/\phi=\sqrt{5}$ and (\ref{resdef1}) and (\ref{ressimp1}):
\begin{equation}\label{mult1}
 \frac{1}{\sqrt{5}} \frac{\phi^a (x^{n+1}\phi^{b(n+1)}-1)(x(-1)^b\phi^{-b}-1)}{1-L_b x+(-1)^b x^2}
\end{equation}
and the residue for $y=-\phi$ is the same with $\phi$ replaced by $1-\phi$ and a minus sign.
Multiplying out this result yields:
\begin{equation}
 \sum_{k=0}^n x^k F_{a+bk} = \frac{x^{n+1}[(-1)^bF_{a+bn}x - F_{a+b(n+1)}] - [(-1)^bF_{a-b}x - F_a]}
  {1-L_b x+(-1)^b x^2}
\end{equation}
From the earlier paper \cite{K18} equation (2.27):
\begin{equation}\label{gfldef}
 G_n = \frac{1}{2} [ ( G_{-1}+G_1 ) F_n + G_0 L_n ]
\end{equation}
and using $F_{a-1}+F_{a+1}=L_a$, it is concluded again that adding the identity for $a-1$ and $a+1$
makes the identity true for $L$ instead of $F$, and therefore for $G$, and so (\ref{res2}) has been derived.\\
A term $x^w/(x-1)^v$ in the modified finite geometric series (\ref{fingeom1}) to (\ref{fingeom5})
results in (omitting a factor $(-1)^v$ which is cancelled by the alternating signs of these terms):
\begin{equation}
  \Delta^{-v} \phi^a x^w \phi^{bw} ( 1 - x(-1)^b\phi^{-b} )^v = \Delta^{-v} x^w \sum_{k=0}^v \binom{v}{k} (-1)^{(b+1)k} x^k \phi^{a+bw-bk}
\end{equation}
and likewise with $\phi$ replaced by $1-\phi$, which results in the polynomial $P(v,w)$ in (\ref{pnmres}).
Given the corresponding modified finite geometric series
(which can be computed with the computer program)
the higher order formula can thus be written down immediately.

\section{Simplification of the Summation Identities}

The resulting formulas (\ref{res2}) to (\ref{res7}) are combinations of the polynomials (\ref{pnmres}):
\begin{equation}
 P(v,w) = \sum_{k=0}^v \binom{v}{k} (-1)^{(b+1)k} x^k G_{a+bw-bk}
\end{equation}
For general $x$ using (\ref{gfldef}) and the Binet formulas (\ref{fibodef}) and (\ref{lucasdef})
and $\phi(1-\phi)=-1$:
\begin{equation}
\begin{split}
 P(v,w) = \frac{1}{2} & \{ [ \frac{1}{\sqrt{5}}( G_{-1} + G_1 ) + G_0 ] \phi^{a+bw} ( 1 - x (1-\phi)^b )^v \\
  & - [ \frac{1}{\sqrt{5}}( G_{-1} + G_1 ) - G_0 ] (1-\phi)^{a+bw} ( 1 - x \phi^b )^v \} \\
\end{split}
\end{equation}
For some rational $x$ the polynomials can be simplified to one or two terms.
From the earlier paper \cite{K18} equations (1.1) to (1.4), where in this case $n=v$, $p=a+bw$ and $q=b$:
\begin{equation}\label{polsimp1}
 \sum_{k=0}^n \binom{n}{k} (-1)^{(q+1)k} \left((-1)^q\frac{F_m}{F_{m-q}}\right)^k G_{p-qk} = (-1)^{nm} \left(\frac{F_{-q}}{F_{m-q}}\right)^n G_{p-nm}
\end{equation}
\begin{equation}\label{polsimp2}
 \sum_{k=0}^n \binom{n}{k} (-1)^{(q+1)k} \left(\frac{F_m}{F_{m+q}}\right)^k G_{p-qk} = \left(\frac{F_{q}}{F_{m+q}}\right)^n G_{p+nm}
\end{equation}
\begin{equation}\label{polsimp3}
\begin{split}
 \sum_{k=0}^n & \binom{n}{k} (-1)^{(q+1)k} \left((-1)^q\frac{L_m}{L_{m-q}}\right)^k G_{p-qk} \\
 & = 5^{\lfloor n/2\rfloor} (-1)^{n(m+1)} \left(\frac{F_{-q}}{L_{m-q}}\right)^n [ G_{p-nm+1} - (-1)^n G_{p-nm-1} ] \\
\end{split}
\end{equation}
\begin{equation}\label{polsimp4}
\begin{split}
 \sum_{k=0}^n & \binom{n}{k} (-1)^{(q+1)k} \left(\frac{L_m}{L_{m+q}}\right)^k G_{p-qk} \\
 & = 5^{\lfloor n/2\rfloor} \left(\frac{F_q}{L_{m+q}}\right)^n [ G_{p+nm+1} - (-1)^n G_{p+nm-1} ] \\
\end{split}
\end{equation}
When an $m$ can be found for which $x=(-1)^qF_m/F_{m-q}$ or $x=F_m/F_{m+q}$ or $x=(-1)^qL_m/L_{m-q}$ or $x=L_m/L_{m+q}$,
then the corresponding formula replaces the polynomials by a single or double term.
For example when $x=1$ and $b=q=1$ then (\ref{polsimp2}) with $m=1$ yields $P(v,w)=G_{a+v+w}$,
and when $x=-1$ and $b=q=2$ then (\ref{polsimp4}) with $m=-1$ yields
$P(v,w)=5^{\lfloor v/2\rfloor}(G_{a+2w-v+1}-(-1)^vG_{a+2w-v-1})$.

\section{Derivation of Summation Identities with\\ Binomial Coefficients}

The following sum is to be evaluated, where an infinite sum can be taken
because the summand is zero for $k>n$:
\begin{equation}\label{binomeq}
\begin{split}
 \sum_{k=0}^n \binom{n}{k} x^k F_{a+bk} & = \sum_{k=0}^{\infty} \res{y} \frac{(1+y)^n}{y^{k+1}} x^k \res{z} \frac{(1+z)^{a+bk}}{(z+\phi)(z-1/\phi)} \\
 & = \res{y}\res{z} \frac{(1+y)^n(1+z)^a}{y(z+\phi)(z-1/\phi)} \sum_{k=0}^{\infty} (\frac{x(1+z)^b}{y})^k \\
 & = \res{z}\res{y} \frac{(1+y)^n(1+z)^a}{(z+\phi)(z-1/\phi)(y-x(1+z)^b)} \\
 & = \res{z} \frac{(1+x(1+z)^b)^n (1+z)^a}{(z+\phi)(z-1/\phi)} \\
\end{split}
\end{equation}
Taking the residues at $z=1/\phi$ and $z=-\phi$, an expression without summation only occurs when
for some $m$, using $1+1/\phi=\phi$:
\begin{equation}
 1 + x \phi^b = A \phi^m
\end{equation}
\begin{equation}
 1 + x (1-\phi)^b = B (1-\phi)^m
\end{equation}
This system of two equations cannot always
be solved, and two conditions are found for when it can be solved.
Adding and subtracting the two equations, the following results:
\begin{equation}
 2 + x L_b = \frac{\sqrt{5}}{2}(A-B) F_m + \frac{1}{2}(A+B) L_m 
\end{equation}
\begin{equation}
 x F_b = \frac{1}{2}(A+B) F_m + \frac{1}{2\sqrt{5}}(A-B) L_m
\end{equation}
This is a 2x2 matrix equation, and inverting this 2x2 matrix with:
\begin{equation}
\left(
\begin{matrix}
 a & b \\
 b & c \\
\end{matrix}
\right)^{-1}
= \frac{1}{b^2-ac}
\left(
\begin{matrix}
 -c & b \\
 b & -a \\
\end{matrix}
\right)
\end{equation}
results in:
\begin{equation}\label{binid1}
 F_m = \frac{1}{AB} [ \frac{1}{2}(A+B) x F_b - \frac{1}{2\sqrt{5}} (A-B) ( 2 + x L_b ) ]
\end{equation}
\begin{equation}\label{binid2}
 L_m = \frac{1}{AB} [ \frac{1}{2}(A+B) ( 2 + x L_b ) - \frac{\sqrt{5}}{2} (A-B) x F_b ]
\end{equation}
Two cases are considered, one where $A=B=C$ and one where $A=-B=D$.
When $A=B=C$, dividing the two equations yields:
\begin{equation}
 \frac{F_m}{L_m} = \frac{x F_b}{2 + x L_b} 
\end{equation}
Multiplying this equation out yields:
\begin{equation}
 x ( F_b L_m - L_b F_m ) = 2 F_m
\end{equation}
Using the well known identity \cite{K18,V89}:
\begin{equation}
 F_n L_m - L_n F_m = 2 (-1)^{n+1} F_{m-n}
\end{equation}
this equation simplifies to:
\begin{equation}\label{crit1}
 (-1)^b x F_{m-b} = - F_m
\end{equation}
When an $m$ can be found that fulfills this equation, then (\ref{binid1}) gives:
\begin{equation}\label{cval}
 C = x \frac{F_b}{F_m}
\end{equation}
When $A=-B=D$, dividing the two equations yields:
\begin{equation}
  \frac{F_m}{L_m} = \frac{2 + x L_b}{5 x F_b}
\end{equation}
Multiplying this equation out yields:
\begin{equation}
 x ( 5 F_b F_m - L_b L_m ) = 2 L_m
\end{equation}
Using the identity \cite{K18}:
\begin{equation}
 5 F_n F_m - L_n L_m = 2 (-1)^{n+1} L_{m-n}
\end{equation}
this equation simplifies to:
\begin{equation}\label{crit2}
 (-1)^b x L_{m-b} = - L_m
\end{equation}
When an $m$ can be found that fulfills this equation, then (\ref{binid2}) gives:
\begin{equation}\label{dval}
 D = x \sqrt{5} \frac{F_b}{L_m}
\end{equation}
For example when $x=2$ and $b=1$ then (\ref{crit1}) is solved with $m=3$ and (\ref{cval}) gives $C=1$
and we have derived:
\begin{equation}
 1 + 2\phi = \phi^3
\end{equation}
\begin{equation}
 1 + 2(1-\phi) = (1-\phi)^3
\end{equation}
and (\ref{binomeq}) yields the known identity (\ref{benjamin1}) \cite{BQ03}.
When $x=1$ and $b=2$ then (\ref{crit1}) cannot be solved, but (\ref{crit2}) can be solved with  $m=1$,
and (\ref{dval}) gives $D=\sqrt{5}$ and we have derived:
\begin{equation}
 1 + \phi^2 = \sqrt{5} \phi
\end{equation}
\begin{equation}
 1 + (1-\phi)^2 = -\sqrt{5} (1-\phi)
\end{equation}

\section{Infinite Series and Generating Functions}

When $n\rightarrow\infty$ then from (\ref{gfldef}) follows that $G_{a+bn}$ is $O(\phi^{|b|n})$,
and from (\ref{pnmres}) follows that $P(v,n)$ is also $O(\phi^{|b|n})$.
In that case in the right sides of (\ref{res2}) to (\ref{res7}) therefore $x^n P(v,n)\rightarrow 0$  when:
\begin{equation}
 |x| < \phi^{-|b|}
\end{equation}
When this condition is fulfilled the following infinite series result from (\ref{res2}) to (\ref{res7}),
with $\Delta$ and $P(v,w)$ defined in (\ref{res1}) and (\ref{pnmres}).
These are also the generating functions of the generalized Fibonacci numbers $G_{a+bn}$.
\begin{equation}\label{genfun1}
 \sum_{k=0}^{\infty} x^k G_{a+bk} = \frac{1}{\Delta} P(1,0)
\end{equation}
\begin{equation}\label{genfun2}
 \sum_{k=0}^{\infty} x^k k G_{a+bk} = \frac{1}{\Delta^2} x P(2,1)
\end{equation}
\begin{equation}
 \sum_{k=0}^{\infty} x^k k^2 G_{a+bk} = \frac{1}{\Delta^2} x P(2,1) + \frac{2}{\Delta^3} x^2 P(3,2)
\end{equation}
\begin{equation}
 \sum_{k=0}^{\infty} x^k k^3 G_{a+bk} = \frac{1}{\Delta^2} x P(2,1) + \frac{6}{\Delta^3} x^2 P(3,2) + \frac{6}{\Delta^4} x^3 P(4,3)
\end{equation}
\begin{equation}
 \sum_{k=0}^{\infty} x^k k^4 G_{a+bk} = \frac{1}{\Delta^2} x P(2,1) + \frac{14}{\Delta^3} x^2 P(3,2) + \frac{36}{\Delta^4} x^3 P(4,3)
   + \frac{24}{\Delta^5} x^4 P(5,4) 
\end{equation}
\begin{equation}
\begin{split}
 \sum_{k=0}^{\infty} x^k k^5 G_{a+bk} = & \frac{1}{\Delta^2} x P(2,1) + \frac{30}{\Delta^3} x^2 P(3,2) + \frac{150}{\Delta^4} x^3 P(4,3) + \frac{240}{\Delta^5} x^4 P(5,4) \\
 & + \frac{120}{\Delta^6} x^5 P(6,5) \\
\end{split}
\end{equation}

\section{Generalized Fibonacci Convolution Identities}

The generating functions are used to find the following generalized Fibonacci convolution identities.
\begin{equation}
 \sum_{k=0}^n G_k G_{n-k} = \frac{1}{5} \{ [ (3n+5)G_0 - (n+1)G_1 ] G_n - [ (n-1)G_0 - 2nG_1 ] G_{n+1} \}
\end{equation}
\begin{equation}
 \sum_{k=0}^n k G_k G_{n-k} = \frac{1}{10} n \{ [ (3n+5)G_0 - (n+1)G_1 ] G_n - [ (n-1)G_0 - 2nG_1 ] G_{n+1} \}
\end{equation}
\begin{equation}
\begin{split}
 \sum_{k=0}^n k^2 G_k G_{n-k} = & \frac{1}{150} \{ [ 3n ( 10n^2+25n+13 ) G_0 - (n+1) ( 10n^2+5n-12 ) G_1 ] G_n \\
  & - [ (n-1) (10n^2-5n-12) G_0 - 2n ( 10n^2-7 ) G_1 ] G_{n+1} \}
\end{split}
\end{equation}
\begin{equation}
 \sum_{k=0}^n (-1)^k G_k G_{n-k} = - \frac{1}{2} (1+(-1)^n) ( G_{-1} G_n - G_0 G_{n+1} )
\end{equation}
\begin{equation}
\begin{split}
 \sum_{k=0}^n (-1)^k k G_k G_{n-k} = & (-1)^{n+1} \frac{1}{2} n ( G_{-1} G_n - G_0 G_{n+1} ) \\
  & + \frac{1}{4} (1-(-1)^n) [ G_2 G_n - (G_{-1}+G_1) G_{n+1} ] \\
\end{split}
\end{equation}
\begin{equation}
\begin{split}
 \sum_{k=0}^n & (-1)^k k^2 G_k G_{n-k} \\
  = & (-1)^n \frac{1}{2} n \{ [ (n-1) G_0 - (n+1) G_1 ] G_n + [ (n-1) G_0 + 2 G_1 ] G_{n+1} \} \\
 & - (1+(-1)^n) ( G_1 G_n - G_0 G_{n+1} ) \\
\end{split}
\end{equation}

\section{List of Examples}

Examples with $x=1$:
\begin{equation}
 \sum_{k=0}^n G_{p+k} = G_{p+n+2} - G_{p+1}
\end{equation}
\begin{equation}
 \sum_{k=0}^n k G_{p+k} = n G_{p+n+2} -G_{p+n+3} + G_{p+3}
\end{equation}
\begin{equation}
 \sum_{k=0}^n k^2 G_{p+k} = n^2 G_{p+n+2} - 2n G_{p+n+3} + G_{p+n+6} - G_{p+6}
\end{equation}
\begin{equation}
\begin{split}
 \sum_{k=0}^n k^3 G_{p+k} = & n^3 G_{p+n+2} - 3n^2 G_{p+n+3} + (3n-1) G_{p+n+6} - 2G_{p+n+8} \\
 & + G_{p+6} + 2G_{p+8} \\
\end{split}
\end{equation}
\begin{equation}
\begin{split}
 \sum_{k=0}^n & k^4 G_{p+k} = n^4 G_{p+n+2} - 4n^3 G_{p+n+3} + 6n^2 G_{p+n+6} \\
 & + 4n ( G_{p+n+7} - 3G_{p+n+8} ) + 9 ( G_{p+n+9} - G_{p+9} ) + 2( G_{p+n+10} - G_{p+10} ) \\
\end{split}
\end{equation}
\begin{equation}
\begin{split}
 \sum_{k=0}^n & k^5 G_{p+k} = n^5 G_{p+n+2} - 5n^4 G_{p+n+3} + 10n^3 G_{p+n+6} + 10n^2 ( G_{p+n+7} - 3G_{p+n+8} ) \\
 & + 5n ( 9G_{p+n+9} + 2G_{p+n+10} ) - 34 ( G_{p+n+11} - G_{p+11} ) - 9 ( G_{p+n+12} - G_{p+12} ) \\
\end{split}
\end{equation}
\begin{equation}
 \sum_{k=0}^n G_{p+2k} = G_{p+2n+1} - G_{p-1}
\end{equation}
\begin{equation}
 \sum_{k=0}^n k G_{p+2k} = n G_{p+2n+1} -G_{p+2n} + G_p
\end{equation}
\begin{equation}
 \sum_{k=0}^n k^2 G_{p+2k} = n^2 G_{p+2n+1} - ( 2n-3 ) G_{p+2n} - 2 G_{p+2n-2} - G_{p+1} - G_{p-1}
\end{equation}
\begin{equation}
\begin{split}
 \sum_{k=0}^n k^3 G_{p+2k} = & n^3 G_{p+2n+1} - 3n^2 G_{p+2n} + 3n ( G_{p+2n+1} + G_{p+2n-1} ) \\
  & - 7 ( G_{p+2n} - G_p ) \\
\end{split}
\end{equation}
\begin{equation}
 \sum_{k=0}^n G_{p+3k} = \frac{1}{2} ( G_{p+3n+2} - G_{p-1} )
\end{equation}
\begin{equation}
 \sum_{k=0}^n k G_{p+3k} = \frac{1}{4} ( 2 n G_{p+3n+2} - G_{p+3n+1} + G_{p+1} )
\end{equation}
\begin{equation}
 \sum_{k=0}^n k^2 G_{p+3k} = \frac{1}{4} ( 2 n^2 G_{p+3n+2} - 2n G_{p+3n+1} + G_{p+3n+2} - G_{p+2} )
\end{equation}
\begin{equation}
 \sum_{k=0}^n G_{p+4k} = \frac{1}{5} ( G_{p+4n+4} - G_{p+4n} - G_p + G_{p-4} )
\end{equation}
\begin{equation}
 \sum_{k=0}^n k G_{p+4k} = \frac{1}{5} [ n ( G_{p+4n+4} - G_{p+4n} ) - G_{p+4n} + G_p ]
\end{equation}
\begin{equation}
 \sum_{k=0}^n G_{p+5k} = \frac{1}{11} ( G_{p+5n+5} + G_{p+5n} - G_p - G_{p-5} )
\end{equation}
\begin{equation}
 \sum_{k=0}^n G_{p+6k} = \frac{1}{4} ( G_{p+6n+3} - G_{p-3} )
\end{equation}
\begin{equation}
 \sum_{k=0}^n G_{p+7k} = \frac{1}{29} ( G_{p+7n+7} + G_{p+7n} - G_p - G_{p-7} )
\end{equation}
Examples with $x=2$:
\begin{equation}
 \sum_{k=0}^n 2^k G_{p+k} = \frac{1}{5} [ 2^{n+1} ( G_{p+n+2} + G_{p+n} ) - G_{p+1} - G_{p-1} ]
\end{equation}
\begin{equation}
 \sum_{k=0}^n 2^k k G_{p+k} = \frac{1}{5} \{ 2^{n+1} [ n ( G_{p+n+2} + G_{p+n} ) - G_{p+n+1} ] + 2 G_{p+1} \}
\end{equation}
\begin{equation}
\begin{split}
 \sum_{k=0}^n 2^k k^2 G_{p+k} = \frac{1}{25} \{ 2^{n+1} [ & 5n^2 ( G_{p+n+2} + G_{p+n} ) - 10n G_{p+n+1} \\
 & + G_{p+n+3} + G_{p+n+5} ] - 2 ( G_{p+3} + G_{p+5} ) \}
\end{split}
\end{equation}
\begin{equation}
 \sum_{k=0}^n 2^k G_{p+2k} = 2^{n+1} G_{p+2n-1} - G_{p-3}
\end{equation}
\begin{equation}
 \sum_{k=0}^n 2^k k G_{p+2k} = 2^{n+1} ( n G_{p+2n-1} - G_{p+2n-4} ) + 2 G_{p-4}
\end{equation}
\begin{equation}
\begin{split}
 \sum_{k=0}^n 2^k k^2 G_{p+2k} = & 2^{n+1} ( n^2 G_{p+2n-1} - 2n G_{p+2n-4} - G_{p+2n-6} + 3 G_{p+2n-5} ) \\
 & + 2 ( G_{p-6} - 3 G_{p-5} ) \\
\end{split}
\end{equation}
\begin{equation}
 \sum_{k=0}^n 2^k G_{p+3k} = \frac{1}{11} [ 2^{n+1} ( G_{p+3n+3} + 2 G_{p+3n} ) - G_p - 2 G_{p-3} ]
\end{equation}
\begin{equation}
 \sum_{k=0}^n 2^k G_{p+4k} = \frac{1}{3} ( 2^{n+1} G_{p+4n+1} - G_{p-3} )
\end{equation}
Examples with $x=3$:
\begin{equation}
 \sum_{k=0}^n 3^k G_{p+k} = \frac{1}{11} [ 3^{n+1} ( 3 G_{p+n} + G_{p+n+1} ) - 3 G_{p-1} - G_p ]
\end{equation}
\begin{equation}
\begin{split}
 \sum_{k=0}^n 3^k k G_{p+k} = & \frac{1}{121} \{ 3^{n+1} [ 11n(3G_{p+n}+G_{p+n+1}) + 3G_{p+n}-10G_{p+n+1} ] \\
  & - 3(3G_p-10G_{p+1}) \} \\
\end{split}
\end{equation}
\begin{equation}
 \sum_{k=0}^n 3^k G_{p+2k} = 3^{n+1} G_{p+2n-2} - G_{p-4}
\end{equation}
\begin{equation}
 \sum_{k=0}^n 3^k G_{p+3k} = \frac{1}{10} [ 3^{n+1} ( 2 G_{p+3n+2} - G_{p+3n+1} ) - 2 G_{p-1} + G_{p-2} ]
\end{equation}
Examples with $x=1/2$:
\begin{equation}
 \sum_{k=0}^n 2^{-k}G_{p+k} = -2^{-n}G_{p+n+3} + 2G_{p+2} 
\end{equation}
\begin{equation}
 \sum_{k=0}^n 2^{-k} k G_{p+k} = -2^{-n} ( n G_{p+n+3} + 2 G_{p+n+5} ) + 2G_{p+5}
\end{equation}
\begin{equation}
\begin{split}
 \sum_{k=0}^n 2^{-k} k^2 G_{p+k} = & - 2^{-n} [ n^2 G_{p+n+3} + 4n G_{p+n+5} + 2 ( G_{p+n+7} + G_{p+n+9} ) ] \\
       & + 2(G_{p+7}+G_{p+9}) \\
\end{split}
\end{equation}
\begin{equation}
\begin{split}
 \sum_{k=0}^n 2^{-k} k^3 G_{p+k} = & - 2^{-n} [ n^3 G_{p+n+3} + 6n^2 G_{p+n+5} + 6n ( 2G_{p+n+9} - G_{p+n+8} ) \\
       & + 2 ( 3G_{p+n+12} + G_{p+n+13} ) ] + 2( 3G_{p+12} + G_{p+13} ) \\
\end{split}
\end{equation}
\begin{equation}
 \sum_{k=0}^n 2^{-k}G_{p+2k} = 2^{-n}G_{p+2n+3} - 2G_{p+1}
\end{equation}
\begin{equation}
 \sum_{k=0}^n 2^{-k} k G_{p+2k} = 2^{-n} ( n G_{p+2n+3} - 2G_{p+2n+4} ) + 2 G_{p+4}
\end{equation}
\begin{equation}
\begin{split}
 \sum_{k=0}^n 2^{-k} k^2 G_{p+2k} = & 2^{-n} [ n^2 G_{p+2n+3} - 4n G_{p+2n+4} + 2 ( 3 G_{p+2n+5} + G_{p+2n+6} ) ] \\
  & - 2 ( 3 G_{p+5} + G_{p+6} ) \\
\end{split}
\end{equation}
\begin{equation}
 \sum_{k=0}^n 2^{-k} G_{p+3k} = \frac{1}{5} [ 2^{-n} ( G_{p+3n+2} + G_{p+3n+4} ) - 2 ( G_{p+1} + G_{p-1} ) ]
\end{equation}
\begin{equation}
 \sum_{k=0}^n 2^{-k} G_{p+4k} = \frac{1}{3} ( 2^{-n} G_{p+4n+3} - 2 G_{p-1} )
\end{equation}
Examples with $x=1/3$:
\begin{equation}
 \sum_{k=0}^n 3^{-k}G_{p+k} = - \frac{1}{5} [ 3^{-n} ( G_{p+n+3} + G_{p+n+1} ) - 3 ( G_{p+2} + G_p ) ] 
\end{equation}
\begin{equation}
 \sum_{k=0}^n 3^{-k} k G_{p+k} = - \frac{1}{5} \{ 3^{-n} [ n ( G_{p+n+3} + G_{p+n+1} ) + 3 G_{p+n+3} ] - 3 G_{p+3} \} 
\end{equation}
\begin{equation}
\begin{split}
 \sum_{k=0}^n 3^{-k} k^2 G_{p+k} = & - \frac{1}{25} \{ 3^{-n} [ 5 n^2 ( G_{p+n+3} + G_{p+n+1} ) + 30 n G_{p+n+3} \\
 & + 3 ( 7 G_{p+n+5} - G_{p+n+4} ) ] - 3 ( 7 G_{p+5} - G_{p+4} ) \} \\
\end{split}
\end{equation}
\begin{equation}
 \sum_{k=0}^n 3^{-k} G_{p+2k} = - 3^{-n} G_{p+2n+4} + 3 G_{p+2} 
\end{equation}
\begin{equation}
 \sum_{k=0}^n 3^{-k} k G_{p+2k} = - 3^{-n} ( n G_{p+2n+4} + 3 G_{p+2n+6} ) + 3 G_{p+6} 
\end{equation}
\begin{equation}
\begin{split}
 \sum_{k=0}^n 3^{-k} k^2 G_{p+2k} = & - 3^{-n} [ n^2 G_{p+2n+4} + 6n G_{p+2n+6} \\
 & + 3 ( 4 G_{p+2n+8} + G_{p+2n+9} ) ] + 3 ( 4 G_{p+8} + G_{p+9} ) \\
\end{split}
\end{equation}
\begin{equation}
 \sum_{k=0}^n 3^{-k} G_{p+3k} = \frac{1}{2} ( 3^{-n} G_{p+3n+4} - 3 G_{p+1} )
\end{equation}
\begin{equation}
 \sum_{k=0}^n 3^{-k} k G_{p+3k} = \frac{1}{4} [ 3^{-n} ( 2n G_{p+3n+4} - 3 G_{p+3n+5} ) + 3 G_{p+5} ]
\end{equation}
\begin{equation}
\begin{split}
 \sum_{k=0}^n 3^{-k} k^2 G_{p+3k} = & \frac{1}{4} \{ 3^{-n} [ 2n^2 G_{p+3n+4} - 6n G_{p+3n+5} \\
 & - 3 ( G_{p+3n+7}-2G_{p+3n+8} ) ] + 3 ( G_{p+7} - 2G_{p+8} ) \} \\
\end{split}
\end{equation}
Examples with $x=-1$:
\begin{equation}
 \sum_{k=0}^n (-1)^k G_{p+k} = (-1)^n G_{p+n-1} + G_{p-2}
\end{equation}
\begin{equation}
 \sum_{k=0}^n (-1)^k k G_{p+k} = (-1)^n ( n G_{p+n-1} + G_{p+n-3} ) - G_{p-3}
\end{equation}
\begin{equation}
 \sum_{k=0}^n (-1)^k k^2 G_{p+k} = (-1)^n ( n^2 G_{p+n-1} + 2n G_{p+n-3} - G_{p+n-6} ) + G_{p-6}
\end{equation}
\begin{equation}
\begin{split}
 \sum_{k=0}^n (-1)^k k^3 G_{p+k} = & (-1)^n [ n^3 G_{p+n-1} + 3n^2 G_{p+n-3} \\
  & - 3(n+1) G_{p+n-6} + 2G_{p+n-7} ] + G_{p-6} + 2G_{p-8} \\
\end{split}
\end{equation}
\begin{equation}
\begin{split}
 \sum_{k=0}^n & (-1)^k k^4 G_{p+k} = (-1)^n [ n^4 G_{p+n-1} + 4n^3 G_{p+n-3} - 6n^2 G_{p+n-6} \\
  & - 4n ( 3 G_{p+n-8} + G_{p+n-7} ) - 2 G_{p+n-10} + 9 G_{p+n-9} ] + 2 G_{p-10} - 9 G_{p-9} \\
\end{split}
\end{equation}
\begin{equation}
\begin{split}
 \sum_{k=0}^n & (-1)^k k^5 G_{p+k} = (-1)^n [ n^5 G_{p+n-1} + 5n^4 G_{p+n-3} - 10n^3 G_{p+n-6} \\
  & - 10n^2 ( 3 G_{p+n-8} + G_{p+n-7} ) -5n ( 2 G_{p+n-10} - 9 G_{p+n-9} ) \\
  & - 9 G_{p+n-12} + 34 G_{p+n-11} ] + 9 G_{p-12} - 34 G_{p-11} \\
\end{split}
\end{equation}
\begin{equation}
 \sum_{k=0}^n (-1)^k G_{p+2k} = \frac{1}{5} [ (-1)^n ( G_{p+2n+2} + G_{p+2n} ) + G_p + G_{p-2} ]
\end{equation}
\begin{equation}
 \sum_{k=0}^n (-1)^k k G_{p+2k} = \frac{1}{5} \{ (-1)^n [ n ( G_{p+2n+2} + G_{p+2n} ) + G_{p+2n} ] - G_p \}
\end{equation}
\begin{equation}
\begin{split}
 \sum_{k=0}^n (-1)^k k^2 G_{p+2k} = \frac{1}{25} \{ (-1)^n [ & 5 n^2 ( G_{p+2n+2} + G_{p+2n} ) + ( 10 n + 1 ) G_{p+2n} \\
  & - 2 G_{p+2n+1} ] + G_{p+1} + G_{p-1} \} \\
\end{split}
\end{equation}
\begin{equation}
\begin{split}
 \sum_{k=0}^n (-1)^k k^3 G_{p+2k} = & \frac{1}{25} \{ (-1)^n [ 5 n^3 ( G_{p+2n+2} + G_{p+2n} ) + 15 n^2 G_{p+2n} \\
  & - 3n ( G_{p+2n+1} + G_{p+2n-1} ) - G_{p+2n} ] + G_p \} \\
\end{split}
\end{equation}
\begin{equation}
 \sum_{k=0}^n (-1)^k G_{p+3k} = \frac{1}{2} [ (-1)^n G_{p+3n+1} + G_{p-2} ]
\end{equation}
\begin{equation}
 \sum_{k=0}^n (-1)^k k G_{p+3k} = \frac{1}{4} [ (-1)^n ( 2 n G_{p+3n+1} + G_{p+3n-1} ) - G_{p-1} ]
\end{equation}
\begin{equation}
 \sum_{k=0}^n (-1)^k k^2 G_{p+3k} = \frac{1}{4} [ (-1)^n ( 2 n^2 G_{p+3n+1} + 2n G_{p+3n-1} - G_{p+3n-2} ) + G_{p-2} ]
\end{equation}
\begin{equation}
 \sum_{k=0}^n (-1)^k G_{p+4k} = \frac{1}{3} [ (-1)^n G_{p+4n+2} + G_{p-2} ]
\end{equation}
\begin{equation}
 \sum_{k=0}^n (-1)^k k G_{p+4k} = \frac{1}{9} [ (-1)^n ( 3 n G_{p+4n+2} + G_{p+4n} ) - G_p ]
\end{equation}
\begin{equation}
 \sum_{k=0}^n (-1)^k G_{p+5k} = \frac{1}{11} [ (-1)^n ( G_{p+5n+5} - G_{p+5n} ) + G_p - G_{p-5} ]
\end{equation}
\begin{equation}
 \sum_{k=0}^n (-1)^k G_{p+6k} = \frac{1}{10} [ (-1)^n ( G_{p+6n+4} + G_{p+6n+2} ) + G_{p-2} + G_{p-4} ]
\end{equation}
\begin{equation}
 \sum_{k=0}^n (-1)^k G_{p+7k} = \frac{1}{29} [ (-1)^n ( G_{p+7n+7} - G_{p+7n} ) + G_p - G_{p-7} ]
\end{equation}
Examples with $x=-2$:
\begin{equation}
 \sum_{k=0}^n (-1)^k 2^k G_{p+k} = (-1)^n 2^{n+1} G_{p+n-2} + G_{p-3}
\end{equation}
\begin{equation}
 \sum_{k=0}^n (-1)^k 2^k k G_{p+k} = (-1)^n 2^{n+1} ( n G_{p+n-2} + G_{p+n-5} ) - 2 G_{p-5}
\end{equation}
\begin{equation}
\begin{split}
 \sum_{k=0}^n (-1)^k 2^k k^2 G_{p+k} = & (-1)^n 2^{n+1} ( n^2 G_{p+n-2} + 2n G_{p+n-5} + G_{p+n-8} - 2 G_{p+n-7} ) \\
 & - 2 ( G_{p-8} - 2 G_{p-7} ) \\
\end{split}
\end{equation}
\begin{equation}
 \sum_{k=0}^n (-1)^k 2^k G_{p+2k} = \frac{1}{11} [ (-1)^n 2^{n+1} ( G_{p+2n+2} + 2 G_{p+2n} ) + G_p + 2 G_{p-2} ]
\end{equation}
\begin{equation}
\begin{split}
 \sum_{k=0}^n (-1)^k 2^k k G_{p+2k} = & \frac{1}{121} \{ (-1)^n 2^{n+1} [ 11n ( G_{p+2n+2} + 2 G_{p+2n} ) \\
  & - 3 G_{p+2n-1} + 10 G_{p+2n} ] + 2 ( 3 G_{p-1} - 10 G_p ) \} \\
\end{split}
\end{equation}
\begin{equation}
 \sum_{k=0}^n (-1)^k 2^k G_{p+3k} = \frac{1}{5} [ (-1)^n 2^{n+1} ( G_{p+3n+1} + G_{p+3n-1} ) + G_{p-2} + G_{p-4} ]
\end{equation}
\begin{equation}
 \sum_{k=0}^n (-1)^k 2^k G_{p+4k} = \frac{1}{19} [ (-1)^n 2^{n+1} ( G_{p+4n+4} + 2 G_{p+4n} ) + G_p + 2 G_{p-4} ]
\end{equation}
Examples with $x=-3$:
\begin{equation}
 \sum_{k=0}^n (-1)^k 3^k G_{p+k} = \frac{1}{5} [ (-1)^n 3^{n+1} ( G_{p+n} + G_{p+n-2} ) + G_{p-1} + G_{p-3} ]
\end{equation}
\begin{equation}
 \sum_{k=0}^n (-1)^k 3^k k G_{p+k} = \frac{1}{5} \{ (-1)^n 3^{n+1} [ n ( G_{p+n} + G_{p+n-2} ) + G_{p+n-3} ] - 3 G_{p-3} \}
\end{equation}
\begin{equation}
 \sum_{k=0}^n (-1)^k 3^k G_{p+2k} = \frac{1}{19} [ (-1)^n 3^{n+1} ( G_{p+2n+2} + 3 G_{p+2n} ) + G_p + 3 G_{p-2} ]
\end{equation}
\begin{equation}
 \sum_{k=0}^n (-1)^k 3^k G_{p+3k} = \frac{1}{2} [ (-1)^n 3^{n+1} G_{p+3n-1} + G_{p-4} ]
\end{equation}
Examples with $x=-1/2$:
\begin{equation}
 \sum_{k=0}^n (-1)^k 2^{-k} G_{p+k} = \frac{1}{5} [ (-1)^n 2^{-n} ( G_{p+n+1} + G_{p+n-1} ) + 2 ( G_p + G_{p-2} ) ]
\end{equation}
\begin{equation}
 \sum_{k=0}^n (-1)^k 2^{-k} k G_{p+k} = \frac{1}{5} \{ (-1)^n 2^{-n} [ n ( G_{p+n+1} + G_{p+n-1} ) + 2 G_{p+n-1} ] - 2 G_{p-1} \}
\end{equation}
\begin{equation}
\begin{split}
 \sum_{k=0}^n (-1)^k 2^{-k} k^2 G_{p+k} = & \frac{1}{25} \{ (-1)^n 2^{-n} [ 5n^2 ( G_{p+n+1} + G_{p+n-1} ) + 20n G_{p+n-1} \\ 
  & + 2 ( G_{p+n-5} + G_{p+n-3} ) ] - 2 ( G_{p-5} + G_{p-3} ) \} \\
\end{split}
\end{equation}
\begin{equation}
 \sum_{k=0}^n (-1)^k 2^{-k} G_{p+2k} = \frac{1}{11} [ (-1)^n 2^{-n} ( G_{p+2n} + 2 G_{p+2n+2} ) + 2 ( G_{p-2} + 2 G_p ) ]
\end{equation}
\begin{equation}
\begin{split}
 \sum_{k=0}^n (-1)^k 2^{-k} k G_{p+2k} = & \frac{1}{121} \{ (-1)^n 2^{-n} [ 11n ( G_{p+2n} + 2 G_{p+2n+2} ) \\
  & + 2 ( 10 G_{p+2n} + 3 G_{p+2n+1} ) ] - 2 ( 10 G_p + 3 G_{p+1} ) \}
\end{split}
\end{equation}
\begin{equation}
 \sum_{k=0}^n (-1)^k 2^{-k} G_{p+3k} = \frac{1}{11} [ (-1)^n 2^{-n} ( 2 G_{p+3n+3} - G_{p+3n} ) + 2 ( 2 G_p - G_{p-3} ) ]
\end{equation}
\begin{equation}
 \sum_{k=0}^n (-1)^k 2^{-k} G_{p+4k} = \frac{1}{19} [ (-1)^n 2^{-n} ( 2 G_{p+4n+4} + G_{p+4n} ) + 2 ( 2 G_p + G_{p-4} ) ]
\end{equation}
Examples with $x=-1/3$:
\begin{equation}
 \sum_{k=0}^n (-1)^k 3^{-k} G_{p+k} = \frac{1}{11} [ (-1)^n 3^{-n} ( 3 G_{p+n+1} - G_{p+n} ) + 3 ( 3 G_p - G_{p-1} ) ]
\end{equation}
\begin{equation}
\begin{split}
 \sum_{k=0}^n (-1)^k 3^{-k} k G_{p+k} = & \frac{1}{121} \{ (-1)^n 3^{-n} [ 11 n ( 3 G_{p+n+1} - G_{p+n} ) \\
 & + 3 ( 10 G_{p+n-1} + 3 G_{p+n} ) ] - 3 ( 10 G_{p-1} + 3 G_p ) \} \\
\end{split}
\end{equation}
\begin{equation}
 \sum_{k=0}^n (-1)^k 3^{-k} G_{p+2k} = \frac{1}{19} [ (-1)^n 3^{-n} ( 3 G_{p+2n+2} + G_{p+2n} ) + 3 ( 3 G_p + G_{p-2} ) ]
\end{equation}
\begin{equation}
 \sum_{k=0}^n (-1)^k 3^{-k} G_{p+3k} = \frac{1}{20} [ (-1)^n 3^{-n} ( 3 G_{p+3n+3} - G_{p+3n} ) + 3 ( 3 G_p - G_{p-3} ) ]
\end{equation}

\section{List of Examples with Binomial Coefficients}

Examples with $x=1$:
\begin{equation}
 \sum_{k=0}^n \binom{n}{k} G_{p+k} = G_{p+2n}
\end{equation}
\begin{equation}
 \sum_{k=0}^n \binom{n}{k} k G_{p+k} = n G_{p+2n-1}
\end{equation}
\begin{equation}
 \sum_{k=0}^n \binom{n}{k} k^2 G_{p+k} = n ( n G_{p+2n-2} + G_{p+2n-3} )
\end{equation}
\begin{equation}
 \sum_{k=0}^n \binom{n}{k} k^3 G_{p+k} = n ( n^2 G_{p+2n-3} + 3n G_{p+2n-4} - G_{p+2n-6} )
\end{equation}
\begin{equation}
\begin{split}
 \sum_{k=0}^n \binom{n}{k} k^4 G_{p+k} = & n [ n^3 G_{p+2n-4} + 6n^2 G_{p+2n-5} - n ( G_{p+2n-9} -2 G_{p+2n-8} ) \\
 & - 3 G_{p+2n-8} - G_{p+2n-7} ] \\
\end{split}
\end{equation}
\begin{equation}
\begin{split}
 \sum_{k=0}^n \binom{n}{k} k^5 G_{p+k} = & n [ n^4 G_{p+2n-5} + 10n^3 G_{p+2n-6} + 5n^2 ( 3 G_{p+2n-9} + G_{p+2n-8} ) \\
 & - 5n ( 5 G_{p+2n-9} + G_{p+2n-8} ) - 2 G_{p+2n-10} + 9 G_{p+2n-9} ] \\
\end{split}
\end{equation}
\begin{equation}
 \sum_{k=0}^n \binom{n}{k} G_{p+2k} = 5^{\lfloor n/2\rfloor} ( G_{p+n+1} - (-1)^n G_{p+n-1} )
\end{equation}
\begin{equation}
 \sum_{k=0}^n \binom{n}{k} k G_{p+2k} = 5^{\lfloor (n-1)/2\rfloor} n ( G_{p+n+2} + (-1)^n G_{p+n} )
\end{equation}
\begin{equation}
\begin{split}
 \sum_{k=0}^n \binom{n}{k} k^2 G_{p+2k} = & n [ 5^{\lfloor (n-1)/2\rfloor} ( G_{p+n+2} + (-1)^n G_{p+n} ) \\
   & + 5^{\lfloor n/2\rfloor-1} (n-1) ( G_{p+n+3} - (-1)^n G_{p+n+1} ) ] \\
\end{split}
\end{equation}
\begin{equation}
\begin{split}
 \sum_{k=0}^n \binom{n}{k} k^3 G_{p+2k} = & n \{ 5^{\lfloor (n-1)/2\rfloor-1} [ 5 ( G_{p+n+2} + (-1)^n G_{p+n} ) \\
   & + (n-1)(n-2) ( G_{p+n+4} + (-1)^n G_{p+n+2} ) ] \\
   & + 5^{\lfloor n/2\rfloor-1} 3(n-1) ( G_{p+n+3} - (-1)^n G_{p+n+1} ) \} \\
\end{split}
\end{equation}
\begin{equation}
 \sum_{k=0}^n \binom{n}{k} G_{p+3k} = 2^n G_{p+2n}
\end{equation}
\begin{equation}
 \sum_{k=0}^n \binom{n}{k} k G_{p+3k} = 2^{n-1} n G_{p+2n+1}
\end{equation}
\begin{equation}
 \sum_{k=0}^n \binom{n}{k} k^2 G_{p+3k} = 2^{n-2} n ( n G_{p+2n+2} + G_{p+2n-1} )
\end{equation}
\begin{equation}
 \sum_{k=0}^n \binom{n}{k} k^3 G_{p+3k} = 2^{n-3} n ( n^2 G_{p+2n+3} + 3n G_{p+2n} - 2 G_{p+2n-2} )
\end{equation}
\begin{equation}
 \sum_{k=0}^n \binom{n}{k} G_{p+4k} = 3^n G_{p+2n}
\end{equation}
\begin{equation}
 \sum_{k=0}^n \binom{n}{k} k G_{p+4k} = 3^{n-1} n G_{p+2n+2}
\end{equation}
\begin{equation}
 \sum_{k=0}^n \binom{n}{k} k^2 G_{p+4k} = 3^{n-2} n ( n G_{p+2n+4} + G_{p+2n} )
\end{equation}
\begin{equation}
 \sum_{k=0}^n \binom{n}{k} k^3 G_{p+4k} = 3^{n-3} n ( n^2 G_{p+2n+6} + 3n G_{p+2n+2} - G_{p+2n+1} - G_{p+2n-1} )
\end{equation}
Examples with $x=2$:
\begin{equation}\label{benjamin1}
 \sum_{k=0}^n \binom{n}{k} 2^k G_{p+k} = G_{p+3n}
\end{equation}
\begin{equation}
 \sum_{k=0}^n \binom{n}{k} 2^k k G_{p+k} = 2n G_{p+3n-2}
\end{equation}
\begin{equation}
 \sum_{k=0}^n \binom{n}{k} 2^k k^2 G_{p+k} = 2n ( 2n G_{p+3n-4} + G_{p+3n-5} )
\end{equation}
\begin{equation}
 \sum_{k=0}^n \binom{n}{k} 2^k k^3 G_{p+k} = 2n [ 4n^2 G_{p+3n-6} + ( 6n-1 ) G_{p+3n-7} - G_{p+3n-9} ]
\end{equation}
Examples with $x=3$:
\begin{equation}
 \sum_{k=0}^n \binom{n}{k} 3^k G_{p+k} = 5^{\lfloor n/2\rfloor} ( G_{p+2n+1} - (-1)^n G_{p+2n-1} )
\end{equation}
\begin{equation}
 \sum_{k=0}^n \binom{n}{k} 3^k k G_{p+k} = 5^{\lfloor (n-1)/2\rfloor} 3 n ( G_{p+2n} + (-1)^n G_{p+2n-2} )
\end{equation}
\begin{equation}
\begin{split}
 \sum_{k=0}^n \binom{n}{k} 3^k k^2 G_{p+k} = & 3n [ 5^{\lfloor (n-1)/2\rfloor} ( G_{p+2n} + (-1)^n G_{p+2n-2} ) \\
 & + 5^{\lfloor n/2\rfloor-1} 3(n-1) (G_{p+2n-1}-(-1)^nG_{p+2n-3}) ] \\
\end{split}
\end{equation}
\begin{equation}
\begin{split}
 \sum_{k=0}^n \binom{n}{k} 3^k k^3 G_{p+k} = & 3n \{ 5^{\lfloor (n-1)/2\rfloor-1} [ 5 ( G_{p+2n} + (-1)^n G_{p+2n-2} ) \\
 & + 9(n-1)(n-2) ( G_{p+2n-2} + (-1)^n G_{p+2n-4} ) ] \\
 & + 5^{\lfloor n/2\rfloor-1} 9(n-1) (G_{p+2n-1}-(-1)^nG_{p+2n-3}) \} \\
\end{split}
\end{equation}
Examples with $x=1/2$:
\begin{equation}
 \sum_{k=0}^n \binom{n}{k} 2^{-k} G_{p+k} = 5^{\lfloor n/2\rfloor} 2^{-n} ( G_{p+n+1} - (-1)^n G_{p+n-1} )
\end{equation}
\begin{equation}
 \sum_{k=0}^n \binom{n}{k} 2^{-k} k G_{p+k} = 5^{\lfloor (n-1)/2\rfloor} 2^{-n} n ( G_{p+n+1} + (-1)^n G_{p+n-1} )
\end{equation}
\begin{equation}
\begin{split}
 \sum_{k=0}^n \binom{n}{k} 2^{-k} k^2 G_{p+k} = & 2^{-n} n [ 5^{\lfloor (n-1)/2\rfloor} ( G_{p+n+1} + (-1)^n G_{p+n-1} ) \\
 & + 5^{\lfloor n/2\rfloor-1} (n-1) ( G_{p+n+1} - (-1)^n G_{p+n-1}) ] \\
\end{split}
\end{equation}
\begin{equation}
\begin{split}
 \sum_{k=0}^n & \binom{n}{k} 2^{-k} k^3 G_{p+k} \\
 = & 2^{-n} n \{ 5^{\lfloor (n-1)/2\rfloor-1} [5+(n-1)(n-2)] ( G_{p+n+1} + (-1)^n G_{p+n-1} ) \\
 & + 5^{\lfloor n/2\rfloor-1} 3 (n-1) ( G_{p+n+1} - (-1)^n G_{p+n-1}) \} \\
\end{split}
\end{equation}
Examples with $x=-1$:
\begin{equation}
 \sum_{k=0}^n \binom{n}{k} (-1)^k G_{p+k} = (-1)^n G_{p-n}
\end{equation}
\begin{equation}
 \sum_{k=0}^n \binom{n}{k} (-1)^k k G_{p+k} = (-1)^n n G_{p-n+2}
\end{equation}
\begin{equation}
 \sum_{k=0}^n \binom{n}{k} (-1)^k k^2 G_{p+k} = (-1)^n n ( n G_{p-n+4} - G_{p-n+3} )
\end{equation}
\begin{equation}
 \sum_{k=0}^n \binom{n}{k} (-1)^k k^3 G_{p+k} = (-1)^n n [ (n^2 + 1) G_{p-n+6} - 3n G_{p-n+5} ]
\end{equation}
\begin{equation}
\begin{split}
 \sum_{k=0}^n \binom{n}{k} (-1)^k k^4 G_{p+k} = (-1)^n n & [ n^3 G_{p-n+8} - 6n^2 G_{p-n+7} \\
  & + (7n-3) G_{p-n+8} - (3n-1) G_{p-n+7} ] \\
\end{split}
\end{equation}
\begin{equation}
\begin{split}
 & \sum_{k=0}^n \binom{n}{k} (-1)^k k^5 G_{p+k} = (-1)^n n [ n^4 G_{p-n+10} - 10n^3 G_{p-n+9} \\
  & + 5n^2 ( 5 G_{p-n+8} + 2 G_{p-n+9} ) - (5n-9) G_{p-n+9} - (15n-2) G_{p-n+10} ] \\
\end{split}
\end{equation}
\begin{equation}
 \sum_{k=0}^n \binom{n}{k} (-1)^k G_{p+2k} = (-1)^n G_{p+n}
\end{equation}
\begin{equation}
 \sum_{k=0}^n \binom{n}{k} (-1)^k k G_{p+2k} = (-1)^n n G_{p+n+1}
\end{equation}
\begin{equation}
 \sum_{k=0}^n \binom{n}{k} (-1)^k k^2 G_{p+2k} = (-1)^n n ( n G_{p+n+2} - G_{p+n} )
\end{equation}
\begin{equation}
 \sum_{k=0}^n \binom{n}{k} (-1)^k k^3 G_{p+2k} = (-1)^n n [ n^2 G_{p+n+3} - (3n-2) G_{p+n+1} - G_{p+n} ]
\end{equation}
\begin{equation}
 \sum_{k=0}^n \binom{n}{k} (-1)^k G_{p+3k} = (-1)^n 2^n G_{p+n}
\end{equation}
\begin{equation}
 \sum_{k=0}^n \binom{n}{k} (-1)^k k G_{p+3k} = (-1)^n 2^{n-1} n G_{p+n+2}
\end{equation}
\begin{equation}
 \sum_{k=0}^n \binom{n}{k} (-1)^k k^2 G_{p+3k} = (-1)^n 2^{n-2} n ( n G_{p+n+4} - G_{p+n+1} )
\end{equation}
\begin{equation}
 \sum_{k=0}^n \binom{n}{k} (-1)^k k^3 G_{p+3k} = (-1)^n 2^{n-3} n ( n^2 G_{p+n+6} - 3n G_{p+n+3} + 2 G_{p+n+2} )
\end{equation}
Examples with $x=-2$:
\begin{equation}
 \sum_{k=0}^n \binom{n}{k} (-1)^k 2^k G_{p+k} = 5^{\lfloor n/2\rfloor} ( (-1)^n G_{p+1} - G_{p-1} )
\end{equation}
\begin{equation}
 \sum_{k=0}^n \binom{n}{k} (-1)^k 2^k k G_{p+k} = 5^{\lfloor (n-1)/2\rfloor} 2n ( (-1)^n G_{p+2} + G_p )
\end{equation}
\begin{equation}
\begin{split}
 \sum_{k=0}^n \binom{n}{k} (-1)^k 2^k k^2 G_{p+k} = & 2n [ 5^{\lfloor (n-1)/2\rfloor} ( (-1)^n G_{p+2} + G_p ) \\
  & + 5^{\lfloor n/2\rfloor-1} 2(n-1) ( (-1)^n G_{p+3} - G_{p+1} ) ] \\
\end{split}
\end{equation}
\begin{equation}
\begin{split}
 \sum_{k=0}^n \binom{n}{k} (-1)^k 2^k k^3 G_{p+k} = & 2n \{ 5^{\lfloor (n-1)/2\rfloor-1} [ 5 ( (-1)^n G_{p+2} + G_p ) \\
  & + 4(n-1)(n-2) ( (-1)^n G_{p+4} + G_{p+2} ) ] \\
  & + 5^{\lfloor n/2\rfloor-1} 6(n-1) ( (-1)^n G_{p+3} - G_{p+1} ) \} \\
\end{split}
\end{equation}
\begin{equation}
 \sum_{k=0}^n \binom{n}{k} (-1)^k 2^k G_{p+2k} = (-1)^n G_{p+3n}
\end{equation}
\begin{equation}
 \sum_{k=0}^n \binom{n}{k} (-1)^k 2^k k G_{p+2k} = (-1)^n 2 n G_{p+3n-1}
\end{equation}
\begin{equation}
 \sum_{k=0}^n \binom{n}{k} (-1)^k 2^k k^2 G_{p+2k} = (-1)^n 2 n ( 2n G_{p+3n-2} - G_{p+3n-4} )
\end{equation}
\begin{equation}
 \sum_{k=0}^n \binom{n}{k} (-1)^k 2^k k^3 G_{p+2k} = (-1)^n 2 n [ 4n^2 G_{p+3n-3} - 3(2n-1) G_{p+3n-5} - G_{p+3n-6} ]
\end{equation}
Examples with $x=-3$:
\begin{equation}
 \sum_{k=0}^n \binom{n}{k} (-1)^k 3^k G_{p+2k} = (-1)^n G_{p+4n}
\end{equation}
\begin{equation}
 \sum_{k=0}^n \binom{n}{k} (-1)^k 3^k k G_{p+2k} = (-1)^n 3 n G_{p+4n-2}
\end{equation}
\begin{equation}
 \sum_{k=0}^n \binom{n}{k} (-1)^k 3^k k^2 G_{p+2k} = (-1)^n 3 n ( 3 n G_{p+4n-4} - G_{p+4n-6} )
\end{equation}
\begin{equation}
\begin{split}
 \sum_{k=0}^n & \binom{n}{k} (-1)^k 3^k k^3 G_{p+2k} \\
 = & (-1)^n 3 n [ 9 n^2 G_{p+4n-6} - (9n-4) G_{p+4n-8} - G_{p+4n-9} ] \\
\end{split}
\end{equation}
Examples with $x=-1/2$:
\begin{equation}
 \sum_{k=0}^n \binom{n}{k} (-1)^k 2^{-k} G_{p+k} = 2^{-n} G_{p-2n}
\end{equation}
\begin{equation}
 \sum_{k=0}^n \binom{n}{k} (-1)^k 2^{-k} k G_{p+k} = - 2^{-n} n G_{p-2n+3}
\end{equation}
\begin{equation}
 \sum_{k=0}^n \binom{n}{k} (-1)^k 2^{-k} k^2 G_{p+k} = 2^{-n} n ( n G_{p-2n+6} - 2 G_{p-2n+5} )
\end{equation}
\begin{equation}
 \sum_{k=0}^n \binom{n}{k} (-1)^k 2^{-k} k^3 G_{p+k} = - 2^{-n} n [ (n^2+4) G_{p-2n+9} - 2(3n+1) G_{p-2n+8} ]
\end{equation}
\begin{equation}
 \sum_{k=0}^n \binom{n}{k} (-1)^k 2^{-k} G_{p+2k} = (-1)^n 2^{-n} G_{p-n}
\end{equation}
\begin{equation}
 \sum_{k=0}^n \binom{n}{k} (-1)^k 2^{-k} k G_{p+2k} = (-1)^n 2^{-n} n G_{p-n+3}
\end{equation}
\begin{equation}
 \sum_{k=0}^n \binom{n}{k} (-1)^k 2^{-k} k^2 G_{p+2k} = (-1)^n 2^{-n} n ( n G_{p-n+6} - 2 G_{p-n+4} )
\end{equation}
\begin{equation}
\begin{split}
 \sum_{k=0}^n & \binom{n}{k} (-1)^k 2^{-k} k^3 G_{p+2k} \\
  = & (-1)^n 2^{-n} n [ n^2 G_{p-n+9} - 6n G_{p-n+7} + 2 ( 3 G_{p-n+5} + G_{p-n+6} ) ] \\
\end{split}
\end{equation}
Examples with $x=-1/3$:
\begin{equation}
 \sum_{k=0}^n \binom{n}{k} (-1)^k 3^{-k} G_{p+k} = 5^{\lfloor n/2\rfloor} 3^{-n} ( G_{p-n+1} - (-1)^n G_{p-n-1} )
\end{equation}
\begin{equation}
 \sum_{k=0}^n \binom{n}{k} (-1)^k 3^{-k} k G_{p+k} = - 5^{\lfloor (n-1)/2\rfloor} 3^{-n} n ( G_{p-n+3} + (-1)^n G_{p-n+1} )
\end{equation}
\begin{equation}
\begin{split}
 \sum_{k=0}^n \binom{n}{k} (-1)^k 3^{-k} k^2 G_{p+k} = & - 3^{-n} n [ 5^{\lfloor (n-1)/2\rfloor} ( G_{p-n+3} + (-1)^n G_{p-n+1} ) \\
 & - (n-1) 5^{\lfloor n/2\rfloor-1} ( G_{p-n+5} - (-1)^n G_{p-n+3} ) ] \\
\end{split}
\end{equation}
\begin{equation}
\begin{split}
 \sum_{k=0}^n & \binom{n}{k} (-1)^k 3^{-k} k^3 G_{p+k} \\
 = & - 3^{-n} n \{ 5^{\lfloor (n-1)/2\rfloor-1} [ 5 ( G_{p-n+3} + (-1)^n G_{p-n+1} ) \\
 & + (n-1)(n-2) ( G_{p-n+7} + (-1)^n G_{p-n+5} ) ] \\
 & - 3 (n-1) 5^{\lfloor n/2\rfloor-1} ( G_{p-n+5} - (-1)^n G_{p-n+3} ) \} \\
\end{split}
\end{equation}
\begin{equation}
 \sum_{k=0}^n \binom{n}{k} (-1)^k 3^{-k} G_{p+2k} = 3^{-n} G_{p-2n}
\end{equation}
\begin{equation}
 \sum_{k=0}^n \binom{n}{k} (-1)^k 3^{-k} k G_{p+2k} = - 3^{-n} n G_{p-2n+4}
\end{equation}
\begin{equation}
 \sum_{k=0}^n \binom{n}{k} (-1)^k 3^{-k} k^2 G_{p+2k} = 3^{-n} n ( n G_{p-2n+8} - 3 G_{p-2n+6} )
\end{equation}
\begin{equation}
\begin{split}
 \sum_{k=0}^n & \binom{n}{k} (-1)^k 3^{-k} k^3 G_{p+2k} \\
 & = - 3^{-n} n [ n^2 G_{p-2n+12} - 9n G_{p-2n+10} + 3 ( 4 G_{p-2n+8} + G_{p-2n+9} ) ] \\
\end{split}
\end{equation}
\begin{equation}
 \sum_{k=0}^n \binom{n}{k} (-1)^k 3^{-k} G_{p+3k} = (-1)^n 2^n 3^{-n} G_{p-n}
\end{equation}
\begin{equation}
 \sum_{k=0}^n \binom{n}{k} (-1)^k 3^{-k} k G_{p+3k} = (-1)^n 2^{n-1} 3^{-n} n G_{p-n+4}
\end{equation}
\begin{equation}
 \sum_{k=0}^n \binom{n}{k} (-1)^k 3^{-k} k^2 G_{p+3k} = (-1)^n 2^{n-2} 3^{-n} n ( n G_{p-n+8} - 3 G_{p-n+5} )
\end{equation}
\begin{equation}
\begin{split}
 \sum_{k=0}^n & \binom{n}{k} (-1)^k 3^{-k} k^3 G_{p+3k} \\
 = & (-1)^n 2^{n-3} 3^{-n} n [ n^2 G_{p-n+12} - 9n G_{p-n+9} + 6( 2 G_{p-n+8} - G_{p-n+7} ) ] \\
\end{split}
\end{equation}

\section{List of Examples of Infinite Series}

\begin{equation}
 \sum_{k=0}^{\infty} 2^{-k} G_{p+k} = 2 G_{p+2}
\end{equation}
\begin{equation}
 \sum_{k=0}^{\infty} 2^{-k} k G_{p+k} = 2 G_{p+5}
\end{equation}
\begin{equation}
 \sum_{k=0}^{\infty} 2^{-k} k^2 G_{p+k} = 2 ( G_{p+7} + G_{p+9} )
\end{equation}
\begin{equation}
 \sum_{k=0}^{\infty} 2^{-k} k^3 G_{p+k} = 2 ( 3 G_{p+12} + G_{p+13} )
\end{equation}
\begin{equation}
 \sum_{k=0}^{\infty} 3^{-k} G_{p+k} = \frac{3}{5} ( G_{p+2} + G_p )
\end{equation}
\begin{equation}
 \sum_{k=0}^{\infty} 3^{-k} k G_{p+k} = \frac{3}{5} G_{p+3}
\end{equation}
\begin{equation}
 \sum_{k=0}^{\infty} 3^{-k} k^2 G_{p+k} = \frac{3}{25} ( 7 G_{p+5} - G_{p+4} )
\end{equation}
\begin{equation}
 \sum_{k=0}^{\infty} 3^{-k} G_{p+2k} = 3 G_{p+2}
\end{equation}
\begin{equation}
 \sum_{k=0}^{\infty} 3^{-k} k G_{p+2k} = 3 G_{p+6}
\end{equation}
\begin{equation}
 \sum_{k=0}^{\infty} 3^{-k} k^2 G_{p+2k} = 3 ( 4 G_{p+8} + G_{p+9} )
\end{equation}
\begin{equation}
 \sum_{k=0}^{\infty} (-1)^k 2^{-k} G_{p+k} = \frac{2}{5} ( G_p + G_{p-2} )
\end{equation}
\begin{equation}
 \sum_{k=0}^{\infty} (-1)^k 2^{-k} k G_{p+k} = -\frac{2}{5} G_{p-1}
\end{equation}
\begin{equation}
 \sum_{k=0}^{\infty} (-1)^k 2^{-k} k^2 G_{p+k} = -\frac{2}{25} ( G_{p-5} + G_{p-3} )
\end{equation}
\begin{equation}
 \sum_{k=0}^{\infty} (-1)^k 3^{-k} G_{p+k} = \frac{3}{11} ( 3 G_p - G_{p-1} )
\end{equation}
\begin{equation}
 \sum_{k=0}^{\infty} (-1)^k 3^{-k} k G_{p+k} = -\frac{3}{121} ( 10 G_{p-1} + 3 G_p )
\end{equation}
\begin{equation}
 \sum_{k=0}^{\infty} (-1)^k 3^{-k} G_{p+2k} = \frac{3}{19} ( 3 G_p + G_{p-2} )
\end{equation}

\pdfbookmark[0]{References}{}

\end{document}